\documentclass{amsart}

\title[Minimal Hypersurfaces with Large Area and Morse Index]{On the Existence of Minimal Hypersurfaces with Arbitrarily Large Area and Morse Index}
\author{Yangyang Li}
\address{Department of Mathematics, Princeton University, Princeton, NJ 08544}
\email{yl15@math.princeton.edu}
\thanks{The author was partially supported by NSF-DMS-1811840.}

\usepackage{bm}
\usepackage{amsfonts,amssymb}
\usepackage{amsmath}
\usepackage{amsthm}
\usepackage{mathrsfs}
\usepackage{graphicx}
\usepackage{tikz}
\usepackage{galois}
\usepackage{soul}
\usepackage[all]{xy}
\usepackage[colorlinks, linkcolor = blue, anchorcolor = blue, citecolor = blue]{hyperref}
\usepackage{eso-pic}

\newtheorem{theorem}{Theorem}

\newtheorem{remark}{Remark}
\newtheorem{lemma}{Lemma}[section]

\newtheorem*{claim}{Claim}

\newcommand{\set}[1]{\{#1\}}

\begin{document}
\bibliographystyle{abbrvalpha}

\begin{abstract}
    We show that a bumpy closed Riemannian manifold $(M^{n+1}, g)$ $(3 \leq n+1 \leq 7)$ admits a sequence of connected closed embedded two-sided minimal hypersurfaces whose areas and Morse indices both tend to infinity. This improves a previous result by O. Chodosh and C. Mantoulidis \cite{chodosh_minimal_2019} on connected minimal hypersurfaces with arbitrarily large area.
\end{abstract}
\maketitle

\section{introduction}

    The Almgren-Pitts min-max theory originated from the work of F.J. Almgren Jr. \cite{almgren_homotopy_1962, almgren_theory_1965} and J. Pitts \cite{pitts_existence_1981} (using a deep regularity result in higher dimensions by Schoen-Simon \cite{schoen_regularity_1981}), recently greatly developed by F.C. Marques and A. Neves \cite{marques_min-max_2014, marques_morse_2016, marques_existence_2017, marques_morse_2018-1}, has presented the abundance of minimal hypersurfaces. In particular, in a closed Riemannian manifold with dimension between $3$ and $7$, by localizing the method in \cite{marques_existence_2017}, A. Song \cite{song_existence_2018} proved that there are infinitely many geometrically distinct embedded minimal hypersurfaces, resolving Yau's conjecture \cite{10.2307/j.ctt1bd6kkq.37}. However, a caveat of this result is the lack of geometric or topological information of these hypersurfaces.

    For generic metrics, the conjecture was proved by Irie-Marques-Neves \cite{irie_density_2018} ($3 \leq n+1 \leq 7$) and the author \cite{li_existence_2019} ($n+1 \geq 8$), using Weyl law developed by Liokumovich-Marques-Neves \cite{liokumovich_weyl_2018}. Moreover, the former result also revealed the denseness of minimal hypersurfaces, and later, Marques-Neves-Song \cite{marques_equidistribution_2017} gave a quantified version of the density, i.e., equidistribution.

    In this paper, we shall focus on the generic case, more precisely, the bumpy metric case. Recall that a closed Riemannian manifold $(M^{n+1}, g)$ is called \textit{bumpy}, if any immersed closed minimal hypersurface is non-degenerate, i.e., has no non-trivial Jacobi field. B. White \cite{white_bumpy_2017} (See also Theorem 9 in \cite{ambrozio_compactness_2017}) has shown that the space of bumpy metrics on $M$ is a residual subset of the space of $C^\infty$ Riemannian metrics on $M$. 

    In a bumpy Riemannian manifold, due to non-degeneracy, minimal hypersurfaces from min-max theory are expected to have better properties. More precisely, in Allen-Cahn min-max setting (proposed by M. Guaraco \cite{guaraco_minmax_2018}, developed in Gaspar-Guaraco \cite{gaspar_weyl_2018}), O. Chodosh and C. Mantoulidis \cite{chodosh_minimal_2018} proved multiplicity one conjecture ($n+1=3$), and more recently, in Almgren-Pitts setting, X. Zhou \cite{zhou_multiplicity_2019}, using the prescribed mean curvature min-max theory developed in Zhou-Zhu \cite{zhou_existence_2018}, confirmed Marques-Neves multiplicity one conjecture ($3\leq n+1 \leq 7$). These results combined with the work of Marques-Neves \cite{marques_morse_2018-1} on the Morse index of $p$-width minimal hypersurfaces lead to:

    \begin{theorem}[\cite{marques_morse_2018-1}, Theorem 8.4]\label{thm:generic}
        On a bumpy closed Riemannian manifold $(M^{n+1}, g)$ $(3 \leq n+1 \leq 7)$, for each $p\in \mathbb{N}^+$, there exists a smooth, closed, embedded, multiplicity one, two-sided, minimal hypersurface $\Sigma_p$ such that
        \begin{equation}
            \omega_p(M, g) = \mathrm{Area}_g(\Sigma_p)\,, \quad \mathrm{index}(\Sigma_p) =p\,,
        \end{equation}
        and
        \begin{equation}
            \lim_{p\rightarrow \infty}\frac{\mathrm{Area}_g(\Sigma_p)}{p^{\frac{1}{n+1}}} = c(n) \mathrm{vol}(M, g)^{\frac{n}{n+1}}\,,
        \end{equation}
        where $\omega_p(M,g)$ is the volume spectrum, more precisely, the min-max $p$-width, and $c(n) > 0$ is a dimensional constant from the Weyl law in \cite{liokumovich_weyl_2018}.
    \end{theorem}

    However, $\Sigma_p$ here might have multiple disjoint components, so it is not obvious whether or not one could find a connected minimal hypersurface with either large Morse index or large area from the theorem. Note that in a manifold with Frankel property, for example, with positive Ricci curvature, $\Sigma_p$ will have only one connected component, so the answer will be trivially yes.

    On the Morse index side, by monotonicity formula, we know that each component has a uniform area lower bound and therefore, the pigeon hole principle will directly imply that for each $\Sigma_p\,$, there exists at least one component with Morse index at least proportional to $p^{n/(n+1)}\,$. Nevertheless, such a component might not have large area.

    On the area side, O. Chodosh and C. Mantoulidis \cite{chodosh_minimal_2019} showed that there exists a sequence of connected minimal hypersurfaces whose areas tend to infinity. 
    \begin{theorem}[Theorem 1.4, \cite{chodosh_minimal_2019}]\label{thm:chodosh-mantoulidis}
        On a bumpy closed Riemannian manifold $(M^{n+1}, g)$ $(3 \leq n+1 \leq 7)\,$, either:
        \begin{itemize}
            \item there exists a sequence of connected closed embedded stable minimal hypersurfaces $\Sigma_j$ with $\mathrm{Area}(\Sigma_j) \rightarrow \infty\,$, or,
            \item the hypersurfaces $\Sigma_p$ from Theorem \ref{thm:generic} have at least one connected component $\Sigma'_p$ with $\mathrm{Area}(\Sigma'_p)\geq Cp^{1/(n+1)}$ for some $C = C(M, g) > 0$ independent of $p\,$.
        \end{itemize}
    \end{theorem}
    Nonetheless, even in the latter case, they couldn't get a lower bound on the Morse index on these components. Hence, it is natural to raise the following question: \\

\noindent\textbf{Q}: \textit{Does any bumpy closed Riemannian manifold admit a connected minimal hypersurface with both arbitrarily large area and Morse index?}\\

    In this paper, we shall give a positive answer to this question.

    \begin{theorem}[Main Theorem]\label{thm:main}
        Given any bumpy closed Riemannian manifold $(M^{n+1}, g)$ $(3 \leq n+1 \leq 7)$, there exists a sequence of connected closed embedded two-sided minimal hypersurfaces $\set{\Gamma_i}$ such that
        \begin{equation}
            \lim_{i \rightarrow \infty} \mathrm{Area}(\Gamma_i) = +\infty\,, \quad \lim_{i \rightarrow \infty} \mathrm{index}(\Gamma_i) = +\infty\,.
        \end{equation}
    \end{theorem}
    \begin{remark}
        When $n + 1 = 3$, many examples of minimal surfaces with bounded topology but without Morse index bounds or area bounds have been constructed in pursuit of answering a question raised by J. Pitts and J. Rubinstein \cite{pittsApplicationsMinimaxMinimal1987}, i.e., whether embedded minimal surfaces of fixed genus in a three-manifold have bounded Morse index. For any closed $3$-manifold, T. Colding and W. Minicozzi \cite{coldingExamplesEmbeddedMinimal1999} constructed a (bumpy) metric admitting stable minimal tori without area bounds. Later, Colding-Hingston \cite{coldingMetricsMorseIndex2003} and Hass-Norbury-Rubinstein \cite{hassMinimalSpheresArbitrarily2003} constructed metrics admitting minimal tori and minimal spheres without Morse index bounds, respectively.
    \end{remark}

    Moreover, the same proof also gives a better description on the growth rate of areas and Morse indices in the latter case of Theorem \ref{thm:chodosh-mantoulidis}, i.e.,
    \begin{theorem}
        In a bumpy closed Riemannian manifold $(M^{n+1}, g)$, if there are at most finitely many stable minimal hypersurfaces, then there exists a sequence of connected minimal hypersurfaces $\Gamma_p$ such that
        \begin{equation}
            \mathrm{index}(\Gamma_p) = p\,,
        \end{equation}
        and their areas satisfy either
        \begin{equation}
            \lim_{p \rightarrow \infty} \frac{\mathrm{Area}(\Gamma_p)}{p^{1/ (n+1)}} = c(n)\,,
        \end{equation}
        or 
        \begin{equation}
            \lim_{p \rightarrow \infty} \frac{\mathrm{Area}(\Gamma_p)}{p} = c_0(M, g) > 0\,.
        \end{equation}
    \end{theorem}

    The idea is based on an adaptation of X. Zhou's multiplicity one theorem \cite{zhou_multiplicity_2019} and Marques-Neves Morse index theorem \cite{marques_morse_2018-1} to a ``core'' manifold introduced by A. Song \cite{song_existence_2018} in the bumpy metric case, which will be described in Section 2. In Section 3, we will divide the proof into three cases based on the number of closed two-sided stable minimal hypersurfaces (none, finitely many and infinitely many), and then apply the adaptive theorems case by case, combined with several counting arguments, to conclude our results.\\

\noindent\textbf{Notation:} In the following, we shall use both $\Sigma$ and $\Gamma$ to denote minimal hypersurfaces, where $\Sigma$ is usually a boundary component of a ``core'' manifold while $\Gamma$ lies inside the interior.

\section*{Acknowledgements}
    
    I would like to thank my advisor Fernando Cod\'a Marques for his constant support and his patient guidance.

\section{Multiplicity and Morse Index for Confined Min-max Closed Minimal Hypersurfaces}

    In this section, we will only consider a Riemannian manifold with boundary $(N, \partial N)\,$, which is a submanifold of the bumpy manifold $(M, g)$, where its boundary $\partial N = \bigcup^m_{i = 1} \Sigma_i$ is a finite union of connected closed embedded two-sided stable minimal hypersurfaces in $M$. 

    \subsection{A. Song's Construction of a Non-compact Manifold with Cylindrical Ends}

        In this subsection, for completeness, we give a brief description of cylindrical ends attached to the minimal hypersurface boundary and their properties (Subection 2.2 in \cite{song_existence_2018}).

        Since the metric of $M$ is bumpy, we know that $\Sigma_i$'s are non-degenerately stable and have a contracting neighborhood in $M$, i.e., there exists a diffeomorphism
        \begin{equation}
            \Phi: \partial N \times [0, \hat t] \rightarrow N\,,
        \end{equation}
        where $\Phi(\partial N \times \set{0}) = \partial N$ is the identity map and for all $t \in [0, \hat t]$, $\Phi(\partial N \times \set{t})$ has non-zero mean curvature vector pointing towards $\partial N$, i.e., has positive mean curvature w.r.t. the direction pointing outwards $\partial N$.

        As in A. Song's paper \cite{song_existence_2018}, we define the corresponding non-compact manifold with cylindrical ends:
        \begin{equation}
            \mathcal{C}(N): = N \cup_\varphi (\partial N \times [0, \infty))\,,
        \end{equation}
        where $\varphi: \partial N \times \set{0} \rightarrow \partial N$ is the canonical identification. The Lipschitz continuous metric $h$ on $\mathcal{C}(N)$ will be given by
        \begin{equation}
            h|_N = g|_N\,, \quad h|_{\partial N \times [0, \infty)} = g|_{\partial N} \oplus \mathrm{d}s^2\,. 
        \end{equation}
        
        Now, for $\varepsilon \in (0, \hat t)\,$, let $N_\varepsilon = N \backslash \Phi(\partial N \times [0, \varepsilon))$ whose boundary is $\partial N_\varepsilon = {\Phi(\partial N \times \set{\varepsilon})}\,$. Due to compactness, there exists a positive number $\hat d > 0$ independent of $\varepsilon$ such that we can define the Fermi coordinates on a $\hat d$-neighborhood of $\partial N_\varepsilon$ as follows:
        \begin{equation}
            \begin{aligned}
                \gamma_\varepsilon: \partial N_\varepsilon \times [- \hat d, \hat d] \rightarrow M\,,\\
                \gamma_\varepsilon(x, t) = \mathrm{exp}(x, t \nu)\,,
            \end{aligned}
        \end{equation}
        where $\mathrm{exp}$ is the exponential map and $\nu$ is the inward unit normal of $\partial N_\varepsilon$. In this Fermi coordinate, $g$ can be written as $g^\varepsilon_t \oplus \mathrm{d}t^2\,$.

        One could choose $\delta_\varepsilon > 0$ small enough so that
        \begin{itemize}
            \item $\lim_{\varepsilon \rightarrow 0} \delta_\varepsilon = 0\,$;
            \item $\gamma_\varepsilon(\partial N_\varepsilon \times [0, \delta_\varepsilon]) \subset \Phi(\partial N \times [\varepsilon, \hat t])\,$;
            \item for all $t \in [0, \delta_\varepsilon]$, $\gamma_\varepsilon(\partial N_\varepsilon \times \set{t})$ has positive mean curvature as well.
        \end{itemize}
        
        One could also find a smooth function $\vartheta_\varepsilon: [0, \delta_\varepsilon] \rightarrow \mathbb{R}$ and a constant $z_\varepsilon \in (0, \delta_\varepsilon)$ satisfying:
        \begin{itemize}
            \item $\vartheta_\varepsilon \geq 1$ and $\frac{\partial}{\partial t} \vartheta_\varepsilon \leq 0\,$;
            \item $\vartheta_\varepsilon \equiv 1$ in a neighborhood of $\delta_\varepsilon$;
            \item $\vartheta_\varepsilon$ is a constant on $[0, z_\varepsilon]\,$;
            \item $\lim_{\varepsilon \rightarrow 0} \int_{[0, \delta_\varepsilon]} \,\vartheta_\varepsilon = \infty\,$;
            \item $\lim_{\varepsilon \rightarrow 0} \int_{[z_\varepsilon, \delta_\varepsilon]} \,\vartheta_\varepsilon = 0\,$.
        \end{itemize}
        
        Abusing the notation, on $N_\varepsilon\,$, let $\vartheta_\varepsilon(\gamma_\varepsilon(x, t)):= \vartheta_\varepsilon(t)$ for all $(x, t) \in \partial N_\varepsilon \times [0, \delta_\varepsilon]\,$.
        
        Now, one could define a new metric $h_\varepsilon$ on $N_\varepsilon\,$:
        \begin{equation}\label{eqn:he}
            h_\varepsilon(q) = \begin{cases}
                g_t(q) \oplus (\vartheta_\varepsilon(q) \mathrm{d}t)^2 & q \in \gamma_\varepsilon(\partial N_\varepsilon \times [0, \delta_\varepsilon])\\
                g(q) & q \in N_\varepsilon\backslash \gamma_\varepsilon(\partial N_\varepsilon \times [0, \delta_\varepsilon])
            \end{cases}.
        \end{equation}

        In the following, we always assume that $N_\varepsilon$ is endowed with the metric $h_\varepsilon\,$.

        Here are some properties on the convergence from $(N_\varepsilon, h_\varepsilon)$ to $(\mathcal{C}(N), h)\,$.

        \begin{lemma}[Lemma 4, \cite{song_existence_2018}]\label{lem:4}
            For $t \in [0, \delta_\varepsilon]\,$, the slices $\gamma_\varepsilon(\partial N_\varepsilon \times \set{t})$ satisfy
            \begin{enumerate}
                \item they have non-zero mean curvature vector pointing in the direction of $-\gamma_{\varepsilon*}\left(\frac{\partial}{\partial t}\right)\,$;
                \item their mean curvature goes uniformly to $0$ as $\varepsilon$ converges to $0\,$;
                \item their second fundamental form is bounded by a constant $C$ independent of $\varepsilon\,$.
            \end{enumerate}
        \end{lemma}

        \begin{lemma}[Lemma 5, \cite{song_existence_2018}]\label{lem:5}
            Let $q \in N \backslash \partial N\,$ such that for small $\varepsilon\,$, $q \in N_\varepsilon\,$. Then $(N_\varepsilon, h_\varepsilon, q)$ converges geometrically to $(\mathcal{C}(N), h, q)$ in the $C^0$ topology. Moreover, the geometric convergence is smooth outside of $\partial N \subset \mathcal{C}(N)$ in the following sense.
            \begin{enumerate}
                \item Let $q \in N\backslash \partial N\,$ such that for small $\varepsilon\,$, $q \in N_\varepsilon\backslash \gamma_\varepsilon(\partial N_\varepsilon \times [0, \delta_\varepsilon))\,$. Then $({N_\varepsilon\backslash \gamma_\varepsilon(\partial N_\varepsilon \times [0, \delta_\varepsilon))}, h_\varepsilon, q)$ converges geometrically to $(N \backslash \partial N, g, q)$ in the $C^\infty$ topology.
                \item Fix any connected component $C$ of $\partial N$; for $\varepsilon$ small enough, we can choose a component $C_\varepsilon$ of $\partial N_\varepsilon$ so that $C_\varepsilon$ converges to $C$ as $\varepsilon \rightarrow 0$. Furthermore, if we let $\set{\varepsilon_k}$ be a sequence converging to $0$, and let $q_k \in \gamma_{\varepsilon_k}(C_{\varepsilon_k}\times [0, \delta_{\varepsilon_k}))$ be a point at any fixed distance $\hat d > 0$ from $\gamma_{\varepsilon_k}(C_{\varepsilon_k} \times \set{\delta_{\varepsilon_k}})$ for the metric $h_{\varepsilon_k}$, $\hat d$ being independent of $k$. Then subsequentially,
                \begin{equation}
                    (\gamma_{\varepsilon_k}(C_{\varepsilon_k} \times [0, \delta_{\varepsilon_k})), h_{\varepsilon_k}, q_k) \rightarrow (C \times (0, \infty), h, q_\infty),
                \end{equation}
                geometrically in the $C^\infty$ topology where $q_\infty \in C\times (0, \infty)$ at a distance $\hat d$ from $C \times \set{0}$.
            \end{enumerate}
        \end{lemma}

        \begin{lemma}[Lemma 6, \cite{song_existence_2018}]\label{lem:6}
            For any $\mathfrak{d}_1 \in (0, \hat d)$, there exists $\eta > 0$ such that for all $\varepsilon > 0$ small, there is an embedding $\theta_\varepsilon: \partial N \times [-\mathfrak{d}_1, \mathfrak{d}_1] \rightarrow N_\varepsilon$ satisfying the following properties:
            \begin{enumerate}
                \item $\theta_\varepsilon(\partial N \times\set{0}) = \gamma_\varepsilon(\partial N_\varepsilon \times \set{\delta_\varepsilon})$ and
                \begin{equation}
                    \theta_\varepsilon(\partial N \times [-\mathfrak{d}_1, \mathfrak{d}_1]) = \set{q \in N_\varepsilon: d_{h_\varepsilon}(q, \gamma_\varepsilon(\partial N_\varepsilon\times \set{\delta_\varepsilon})) \leq \mathfrak{d}_1};
                \end{equation}
                \item $\|\theta^*_\varepsilon h_\varepsilon\|_{C^1(\partial N \times [-\mathfrak{d}_1, \mathfrak{d}_1])} \leq \eta$ where $\|\cdot\|_{C^1(\partial N \times [-\mathfrak{d}_1, \mathfrak{d}_1])}$ is computed w.r.t. the product metric $h':=g\llcorner\partial N \oplus \mathrm{d}s^2$;
                \item the metrics $\theta^*_\varepsilon h_\varepsilon$ converge in the $C^0$ topology to a Lipschitz continuous metric which is smooth outside of $\partial N \times\set{0}$ and
                \begin{equation}
                    \lim_{\varepsilon \rightarrow 0} \|\theta^*_\varepsilon h_\varepsilon\llcorner(\partial N \times [0, \mathfrak{d}_1]) - \gamma^*_0 g \llcorner(\partial N \times [0, \mathfrak{d}_1])\|_{C^0(\partial N \times [0, \mathfrak{d}_1])} = 0.
                \end{equation}
            \end{enumerate}
        \end{lemma}

    \subsection{Improved Confined Min-max Theory}

    In $\mathcal{C}(N)$, one is able to define the min-max $p$-width $\omega_p(\mathcal{C}(N))$ with $K_1 \subset K_2 \subset \cdots \subset \mathcal{C}(N)$ an exhaustion of $\mathcal{C}(N)$ by:
    \begin{equation}
        \omega_p(\mathcal{C}(N)) = \omega_p(\mathcal{C}(N), h) := \lim_{i \rightarrow \infty} \omega_p(K_i, h) \in [0, \infty]\,.
    \end{equation}

    By Theorem 8 in \cite{song_existence_2018}, suppose that $\Sigma_1$ has the largest area among all the boundary components and we know that
    \begin{equation}
        p \cdot \mathrm{Area}(\Sigma_1) \leq \omega_p(\mathcal{C}(N)) \leq p \cdot \mathrm{Area}(\Sigma_1) + \hat C p^{\frac{1}{n+1}}\,,
    \end{equation}
    where $\hat C \in (0, \infty)$ depends on $h$. In addition, By Lemma \ref{lem:5} and Lemma \ref{lem:6},
    \begin{equation}
        \lim_{\varepsilon \rightarrow 0} \omega_p(N_\varepsilon, h_\varepsilon) = \omega_p(\mathcal{C}(N))\,.
    \end{equation}

    Theorem 9 in \cite{song_existence_2018} asserts that such a $p$-width could be realized by a union of minimal hypersurfaces in the interior of $N$. Since the metric of $M$ is bumpy, we could improve it using the multiplicity one theorem by X. Zhou \cite{zhou_multiplicity_2019} and the Morse index theorem by F. Marques and A. Neves \cite{marques_morse_2018-1}.

    \begin{theorem}\label{thm:conf_minmax}
        In $(\mathcal{C}(N), h)$, for any $p \in \mathbb{N}^+$, there exists a pairwise disjoint union of connected, smooth, closed, embedded, two-sided, minimal hypersurfaces $\Gamma_1, \cdots, \Gamma_l$ contained in $N\backslash \partial N$ such that
        \begin{equation}
            \begin{aligned}
                \omega_p(\mathcal{C}(N)) &= \sum^l_{j = 1} \mathrm{Area}(\Gamma_j)\,,\\
                p &= \sum^l_{j = 1} \mathrm{index}(\Gamma_j)\,.
            \end{aligned}
        \end{equation}
    \end{theorem}

    \begin{proof}
        The proof will consist of three steps.\\

        \paragraph*{\textbf{Step 1}} We are going to show that $\exists\varepsilon_0 > 0$ s.t. $\forall \varepsilon\in (0, \varepsilon_0)\,$, any connected embedded minimal hypersurface in $N_\varepsilon$ with area bounded by $\left(\omega_p(\mathcal{C}(N)) + 1 \right)$ from above and index bounded by $p$ from above is inside $N_\varepsilon \backslash \gamma_\varepsilon({\partial N_\varepsilon \times [0, \delta_\varepsilon]})\,$. In fact, this essentially follows the proof for Theorem 9 in \cite{song_existence_2018} so we only sketch it here.

        Note that the metric $h_\varepsilon$ defined in (\ref{eqn:he}) indicates that the region $\gamma_\varepsilon(\partial N_\varepsilon \times [0, \delta_\varepsilon])$ looks more and more cylindrical, and each slice has non-zero mean curvature by Lemma \ref{lem:4} (1). If we fixed a point $q \in N \backslash \partial N$, then for $\varepsilon > 0$ small enough, then we know that $q \in N_\varepsilon \backslash \gamma_\varepsilon(\partial N_\varepsilon \times [0, \delta_\varepsilon])\,$. 

        First, the monotonicity formula together with the maximum principle implies that there exists an $\tilde R > 0$ independent of $\varepsilon$ such that any free boundary (possibly closed) embedded minimal hypersurface $\Gamma$ in $N_\varepsilon$ with area bounded by $\left(\omega_p(\mathcal{C}(N)) + 1\right)$ is inside $B_{h_\varepsilon}(q, \tilde R)$ as long as $\varepsilon>0$ is small enough. Thus, $\Gamma$ must be closed.

        Second, for any sequence of embedded minimal hypersurfaces $\set{\Gamma_{\varepsilon_k} \subset N_{\varepsilon_k}}$ with area bounded by $\left(\omega_p(\mathcal{C}(N), h) + 1\right)$ and Morse index bounded by $p$, if $\varepsilon_k \rightarrow 0\,$, then by Lemma \ref{lem:5} and Lemma \ref{lem:6}, we have a limiting varifold $V_\infty$ in $\mathcal{C}(N)$ from a subsequence of $\set{\|\Gamma_{\varepsilon_k}\|}$ with total volume $\limsup_{k \rightarrow \infty} \mathrm{Area}(\Gamma_{\varepsilon_k})$ bounded by $(\omega_p(\mathcal{C}(N), h) + 1)$. By Sharp's compactness \cite{sharpCompactnessMinimalHypersurfaces2017}, $\mathrm{spt}(V_\infty)$ is smooth outside $\partial N \times \set{0}\,$. The maximum principle implies that if $\mathrm{spt}(V_\infty) \cap (\mathcal{C}(N)\backslash \mathrm{clos}(N)) \neq \emptyset\,$, $\mathrm{spt}(V_\infty)$ would be a connected component of some slice $\partial N \times \set{\delta}\,$. However, due to the mean concaveness of the foliation, $\Gamma_{\varepsilon_k}$ for $k$ large intersects $\Phi(\partial N \times \set{\hat d})$ and so does $\mathrm{spt}(V_\infty)$, giving a contradiction.

        As a consequence, $\mathrm{spt}(V_\infty)$ is contained in $(N, g)\,$. A delicate and direct computation in the proof for Theorem 9 in \cite{song_existence_2018} shows that $V_\infty$ is in fact $g$-stationary, and the maximum principle by B. White \cite{white_maximum_2009} implies that $V_\infty$ is contained in the interior of $N\,$. And therefore, $\Sigma_\infty = \mathrm{spt}(V_\infty)$ is a smooth minimal hypersurface embedded in $N$ with area bounded by $(\omega_p(\mathcal{C}(N), h) + 1)$ and Morse index bounded by $p\,$.

        Finally, since $\partial N$ has no Jacobi field due to bumpiness, by Sharp's compactness theorem again, $\exists \mathfrak{d}_2 > 0$ s.t. any closed embedded minimal hypersurface in the interior of  $N$ with area bounded by ${(\omega_p(\mathcal{C}(N), h) + 1)}$ and Morse index bounded by $p$ has a distance at least $\mathfrak{d}_2$ from $\partial N\,$. Because the convergence above holds for any sequence with $\varepsilon_k \rightarrow 0\,$, and $\delta_\varepsilon \rightarrow 0$ as $\varepsilon \rightarrow 0\,$, the monotonicity formula and the maximum principle together imply that $\exists\varepsilon_0 > 0$ s.t. $\forall \varepsilon\in (0, \varepsilon_0)$, any connected embedded minimal hypersurface in $N_\varepsilon$ with area bounded by $\left(\omega_p(\mathcal{C}(N), g) + 1 \right)$ from above and index bounded by $p$ from above has a distance at least $\mathfrak{d}_2/2 > 0$ away from $\gamma_\varepsilon(\partial N_\varepsilon \times \set{\delta_\varepsilon})$. Hence, the minimal hypersurface is inside $N_\varepsilon \backslash \gamma_\varepsilon({\partial N_\varepsilon \times [0, \delta_\varepsilon]})\,$.
        
        A byproduct is that all these minimal hypersurfaces in $(N_\varepsilon, h_\varepsilon)$ with bounded area and bounded Morse index are in fact minimal hypersurfaces in the interior of $(N, g)$ so non-degenerate and only finitely many.\\

        \paragraph*{\textbf{Step 2}} Fix $\varepsilon \in (0, \varepsilon_0)$ small enough such that $\omega_p(N_\varepsilon, h_\varepsilon) < \omega_p(\mathcal{C}(N)) + 1\,$. Following the proof for Theorem 8.3 in \cite{marques_morse_2018-1}, we can perturb metric $h_\varepsilon$ as in Proposition 8.6 therein to get $h^i_\varepsilon$, and prove for each $h^i_\varepsilon$ the existence of a sequence of two-sided $p$-width minimal hypersurfaces with multiplicity one and index $p\,$. Then, taking the limit and using the nondegeneracy, we obtain a two-sided minimal hypersurface with multiplicity one, index $p$ and area $\omega_p(N_\varepsilon, h_\varepsilon)\,$.

        To be more precise, let $\mathcal{M}_p(h_\varepsilon) = \set{\Gamma_1, \cdots, \Gamma_q}\,,\, q \in \mathbb{N}$ to be the finite set of all connected embedded minimal hypersurfaces in the interior of $N_\varepsilon$ with area bounded by $\left(\omega_p(\mathcal{C}(N)) + 1 \right)$ and index bounded by $p\,$. We choose $\eta> 0$ small and for each $l \in \set{1, \cdots, q}\,$, pick $p_l \in \Gamma_l\backslash(\bigcup_{k \neq l} \Gamma_k)$ s.t. $B_{h_\varepsilon}(p_l, \eta) \cap (\bigcup_{k \neq l} \Gamma_k) = \emptyset$ and nonnegative functions $f_l \in C^\infty_c(B_{h_\varepsilon}(p_l, \eta)), f_l|_{\Gamma_l} \not\equiv 0\,$, s.t., $(\nabla_{h_\varepsilon} f_l)(x) \in T_x \Gamma_l$ for every $x \in \Gamma_l\,$. Then for the metric
        \begin{equation}
            \hat h_\varepsilon(t_1, t_2, \cdots, t_q) = \exp(2\sum^q_{l = 1} t_l f_l) h_\varepsilon\,,
        \end{equation}
        $\Gamma_l$ is still minimal for each $l\,$.
        
        It is possible to take $(t^{(i)}_1, \cdots, t^{(i)}_q) \subset (0, 1)^q$ to be a sequence tending to $0$ and for the metric $h^i_\varepsilon:=\hat h_\varepsilon(t^{(i)}_1, \cdots, t^{(i)}_q)\,$,
        \begin{equation}
            \mathrm{Area}_{h^{i}_\varepsilon}(\Gamma_1), \cdots, \mathrm{Area}_{h^{i}_\varepsilon}(\Gamma_q)\,,
        \end{equation}
        are linearly independent over $\mathbb{Q}$. Sharp's compactness theorem shows that for $i$ large enough, any connected embedded minimal hypersurface in $(N_\varepsilon, h^i_\varepsilon)$ with area bounded by $\left(\omega_p(\mathcal{C}(N), h) + 1 \right)$ and index bounded by $p$ is one of $\set{\Gamma_1, \cdots, \Gamma_q}$ and thus non-degenerate, closed, contained inside $N_\varepsilon \backslash \gamma_\varepsilon({\partial N_\varepsilon \times [0, \delta_\varepsilon]})\,$.

        Proposition 8.7 in \cite{marques_morse_2018-1} still implies the existence of a homotopy class $\Pi_i$ of $p$-sweepouts with $\omega_p(N_\varepsilon, h^i_\varepsilon) = \mathbf{L}(\Pi_i)\,$, and one can observe (See details in Appendix \ref{app:A}) that the proof for X. Zhou's Multiplicity One Theorem \cite{zhou_multiplicity_2019} works in our setting and induces a minimizing sequence $\set{\Phi^{(i)}_j}_j$ in $\Pi_i$ with critical set $\mathbf{C}(\set{\Phi^{(i)}_j}_j)$ containing a multiplicity one, two-sided, embedded $h^i_\varepsilon$-minimal hypersurfaces $\set{\Gamma^{(i)}_j}^l_{j = 1}$ with $\sum^l_{j = 1} \mathrm{index}(\Gamma^{(i)}_j)\leq p$. Furthermore, $\set{\Gamma^{(i)}_j}^l_{j = 1}$ is a subset of $\set{\Gamma_1, \cdots, \Gamma_q}$ mentioned above. The $\mathbb{Q}$-independency of 
        \begin{equation}
            \set{\mathrm{Area}_{h^{i}_\varepsilon}(\Gamma_1), \cdots, \mathrm{Area}_{h^{i}_\varepsilon}(\Gamma_q)}\,,
        \end{equation}
        implies that such a realization is unique. 

        Since all minimal hypersurfaces with bounded area $(\omega_p(\mathcal{C}(N), h) + 1)$ are far away from the boundary $\partial N_\varepsilon$ and closed, Theorem 8.2 in \cite{marques_morse_2018-1} also works in our setting and leads to
        \begin{equation}
            \sum^l_{j = 1} \mathrm{index}(\Gamma^{(i)}_j) = p\,.
        \end{equation}
        
        Taking the limit $i \rightarrow \infty\,$, Sharp's compactness theorem indicates the existence of a two-sided minimal hypersurface with multiplicity one, index $p$ and area $\omega_p(N_\varepsilon, h_\varepsilon)$ (possibly with multiple components) inside $N_\varepsilon \backslash \gamma_\varepsilon({\partial N_\varepsilon \times [0, \delta_\varepsilon]})\,$.\\

        \paragraph*{\textbf{Step 3}} As mentioned in the end of Step 1, the minimal hypersurface from Step 2 is in fact a $g$-minimal hypersurface in the interior of $(N, g)$ with a distance to the boundary at least $\mathfrak{d}_2\,$. Take the limit $\varepsilon \rightarrow 0\,$, Sharp's compactness theorem again leads to the desired union of connected minimal hypersurfaces.
    \end{proof}

\section{Proof of Theorem \ref{thm:main}}

    Firstly, we construct $(N, \partial N)$ by cutting out along two-sided stable minimal hypersurfaces. Indeed, suppose that $\set{\Sigma_i}^I_{i = 1}$ is the union of all connected two-sided stable minimal hypersurfaces in $M$, and since $g$ is bumpy, $I$ is either a natural number or the $\omega$ ordinal. Thus, we may assume that $\mathrm{Area}(\Sigma_i) \leq \mathrm{Area}(\Sigma_{i+1})$ for any $i < I$, if $I \geq 2$.

    If $I \neq 0\,$, $(N, \partial N)$ will be constructed inductively.

    \begin{itemize}
        \item Let $N_1$ be the metric completion of a connected component of $M \backslash \Sigma_1\,$. 
        \item Suppose that $N_j (j \geq 1)$ is constructed. If there exists at least one hypersurface in $\set{\Sigma_i}^I_{i = 1}$ in the interior of $N_j\,$, we may take the first such minimal hypersurface, denoted by $\Sigma'_{j+1}$, and let $N_{j+1}$ be the metric completion of a connected component of $N_j \backslash \Sigma'_{j+1}\,$. Otherwise, we will terminate the induction with $(N, \partial N):=(N_j, \partial N_j)\,$.
    \end{itemize}
    Thus, there will be three cases. If $I = 0\,$, we have $M$ itself, or if $I \neq 0$, we obtain either $(N, \partial N)$ or a sequence of connected manifolds with boundary $\set{(N_j, \partial N_j)}\,$.

    $\,$

    \paragraph*{\textbf{Case 1}} $I = 0\,$.

    We know that $M$ has Frankel property for connected two-sided minimal hypersurfaces, i.e., any two connected two-sided minimal hypersurfaces should intersect with each other. Indeed, suppose this is not true and we have two connected two-sided minimal hypersurfaces $\Gamma_1$ and $\Gamma_2\,$, which must be unstable by assumption. By a standard area-minimization argument (See, for example, the proof of Lemma 11 in \cite{song_existence_2018}), there exists a two-sided stable minimal hypersurface in $M$ giving a contradiction to the assumption that $I = 0\,$.

    Combining Frankel property and Theorem 8.3 in \cite{marques_morse_2018-1}, we have a sequence of connected, smooth, closed, embedded, multiplicity one, two-sided, minimal hypersurfaces $\set{\Gamma_p}^\infty_{p = 1}$ such that 
    \begin{equation}
        \omega_p(M, g) = \mathrm{Area}(\Gamma_p)\,, \quad \mathrm{index}(\Gamma_p) = p\,.
    \end{equation}

    $\,$

    \paragraph*{\textbf{Case 2}} $(N, \partial N) = (N_j, \partial N_j)$ for some $j\,$.

    By Lemma 11 in \cite{song_existence_2018}, we know that $(N, \partial N)$ has Frankel property for two-sided minimal hypersurfaces in the interior. Suppose that $\partial N = \bigcup^m_{i = 1} \Sigma_i$ and $\Sigma_1$ is the one with the largest area. Combining Frankel property and Theorem \ref{thm:conf_minmax}, we have a sequence of connected, smooth, closed, embedded, two-sided minimal hypersurfaces $\set{\Gamma_p}^\infty_{p = 1}$ such that 
    \begin{equation}
        \mathrm{Area}(\Gamma_p) = \omega_p(\mathcal{C}(N))\,, \quad \mathrm{index}(\Gamma_p) = p\,.
    \end{equation}

    $\,$

    \paragraph*{\textbf{Case 3}} There exists an infinite sequence $\set{(N_j, \partial N_j)}\,$.

    For each $(N_j, \partial N_j)\,$, suppose that $\partial N_j = \bigcup^{m_j}_{i = 1} \Sigma^j_i$ where $\Sigma^j_i$ is a connected stable minimal hypersurface and $\Sigma^j_1$ is the one with the largest area. Since the metric $g$ is bumpy, we have that
    \begin{equation}
        \lim_{j \rightarrow \infty}\mathrm{Area}(\Sigma^j_1) = \infty\,.
    \end{equation}
    
    Thus, for any fixed $p \in \mathbb{N}^+\,$, there exists a $j_p \in \mathbb{N}$ such that $\mathrm{Area}(\Sigma^{j_p}_1) \geq p$ and by construction, we know that any two-sided stable minimal hypersurface in the interior of $N_{j_p}$ has area no less than $\mathrm{Area}(\Sigma^{j_p}_1)\,$. In fact, this is true for any two-sided minimal hypersurface inside $N_{j_p}\,$.

    \begin{lemma}\label{lem:area}
        If $\Gamma$ is a connected two-sided minimal hypersurface in $N_{j_p}\backslash \partial N_{j_p}\,$, then $\mathrm{Area}(\Gamma) \geq \mathrm{Area}(\Sigma^{j_p}_1)\,$.
    \end{lemma}
    \begin{proof}
        We only need to deal with the case where $\Gamma$ is unstable. In this case, $\Gamma$ has an expanding neighborhood nearby inside $N_{j_p}\,$. Suppose that the conclusion does not hold, and then the two-sided unstable minimal hypersurface $\Gamma$ has $\mathrm{Area}(\Gamma) < \mathrm{Area}(\Sigma^{j_p}_1)\,$. 

        If $\Gamma$ is non-separated in $N_{j_p}$, we can minimize the area in the cohomology class $[\Gamma] \in H_n(M; \mathbb{Z}_2)$ inside $N_{j_p}\,$. By maximum principle, we obtain a two-sided stable minimal hypersurface with at least one component $\Gamma'$ in the interior of $N_{j_p}$ with
        \begin{equation}
            \mathrm{Area}(\Gamma') < \mathrm{Area}(\Gamma) < \mathrm{Area}(\Sigma^{j_p}_1)\,,
        \end{equation}
        giving a contradiction.
        
        On the other hand, if $\Gamma$ separates $N_{j_p}\,$, let $\tilde N$ be the metric completion of the connected component of $N_{j_p} \backslash \Gamma$ such that $\Sigma^{j_p}_1 \subset \partial \tilde N\,$. Since $\Gamma$ is unstable, there exists an expanding neighborhood of $\Gamma$ in $\tilde N$ and then minimizing the homology class $[\Gamma]$ gives us a stable minimal hypersurface, whose connected component by maximum principle again is either contained in $\partial \tilde N \backslash \Gamma$ or in the interior of $\tilde N\,$. Note that the area of each component is strictly less than $\mathrm{Area}(\Sigma^{j_p}_1)\,$, so none of them could be inside the interior of $N_{j_p}$ and thus the interior of $\tilde N$. As a consequence, by the fact that $\partial \tilde N$ is the boundary of $\tilde N\,$, the minimal hypersurface obtained from the minimizing procedure in $[\Gamma]$ could only be exactly $\partial \tilde N\backslash \Gamma$, which also gives a contradiction from the following inequality:
        \begin{equation}
            \mathrm{Area}(\partial \tilde N\backslash \Gamma) < \mathrm{Area}(\Gamma) < \mathrm{Area}(\Sigma^{j_p}_1) \leq \mathrm{Area}(\partial \tilde N\backslash \Gamma)\,.
        \end{equation}
        
        In summary, any two-sided minimal hypersurface in $N_{j_p}$ has area no less than $\mathrm{Area}(\Sigma^{j_p}_1)\,$.
    \end{proof}

    Now, Let's define
    \begin{equation}
        \begin{aligned}
            \mathcal{N}(p) := \# \set{\Sigma \subset M: \text{ a connected closed two-sided stable minimal hypersurface with }\\
                 \mathrm{Area}(\Sigma) \leq (p+1)\cdot\mathrm{Area}(\Sigma^{j_p}_1)} < \infty\,.
        \end{aligned}
    \end{equation}
    The finiteness follows from the bumpiness of the metric (See Lemma 2.1 in \cite{chodosh_minimal_2019}).

    Following the idea in the proof of Lemma 2.2 in \cite{chodosh_minimal_2019}, we have the following counting estimate.

    \begin{lemma}\label{lem:num}
        Suppose that $\set{\Gamma_k}^l_{k = 1} \subset N_{j_p}\backslash \partial N_{j_p}$ is a set of disjoint, connected, closed, two-sided, minimal hypersurfaces with area no greater than ${(p+1) \cdot \mathrm{Area}(\Sigma^{j_p}_1)}\,$, and then we have
        \begin{equation}
            l \leq 2 \mathcal{N}(p) + 1\,.
        \end{equation}
    \end{lemma}
    \begin{proof}
        Consider $\tilde N = N_{j_p} \backslash \bigcup^l_{k = 1} \Gamma_k$ which can be written as a disjoint union of $\set{\tilde N_t}^Q_{t = 1}\,$, where $(\tilde N_t, \partial \tilde N_t)$ is a connected submanifold with boundary. Note that $Q \leq l + 1\,$, since each cut would produce at most one more connected component. 

        Let $C_t = \#\set{\text{connected components of }\partial \tilde N_t}$ and it is apparent that
        \begin{equation}
            \sum^Q_{t = 1} C_t \geq 2l\,.
        \end{equation}

        Since each two-sided stable minimal hypersurface in $M$ will be counted at most twice, the claim below implies
        \begin{equation}
            2\mathcal{N}(p) \geq \sum^Q_{t = 1} (C_t - 1) = \sum^Q_{t = 1} C_t - Q \geq l - 1\,,
        \end{equation}
        where the last inequality follows from $Q \leq l + 1$.

        \begin{claim}
             For each $\tilde N_t\,$, there are at least $(C_t - 1)$ two-sided stable minimal hypersurfaces in the metric completion of $\tilde N_t$ with area no greater than $(p+1)\cdot \mathrm{Area}(\Sigma^{j_p}_1)\,$.
        \end{claim}
        Indeed, the proof is similar to that of Claim 2.3 in \cite{chodosh_minimal_2019}, but for completeness, we will present it here.
        \begin{proof}[Proof of Claim]
            Let's prove it inductively on the number $n_t$ of unstable components of $\partial \tilde N_t$.

            If $n_t = 0$ or $1$, then the claim holds trivially.

            Let $k \in \mathbb{N}^+$, and suppose that the claim holds for any $n_t \leq k$. Now let's consider the case $n_t = k + 1$, and denote the components of $\partial \tilde N_t$ as $\Gamma_1, \cdots, \Gamma_{n_t}$ which are unstable and $\Gamma'_{n_t+1}, \cdots, \Gamma'_{C_t}$ which are stable. Note that $H_n(\tilde N_t; \mathbb{Z}_2)\ni[\Gamma_{1}] \neq 0$ since $C_t \geq n_t \geq 2\,$.
            
            For the unstable component $\Gamma_1$, we can find $\Gamma'_1 \in [\Gamma_1] \neq 0$ minimizing area in the homology class, each of whose components by maximum principle is either in the interior of $\tilde N_t$ or one stable component of $\partial \tilde N_t\,$. We may denote the number of the components of $\Gamma'_1$ in the interior by $m_t$ and the number of those on the boundary by $l_t$. Moreover, there exists a closed set $\Omega_1 \subset \tilde N_t$ with $\partial (\tilde N_t\backslash \Omega_1) = \Gamma'_1 \cup \Gamma_1\,$. Note that the number of connected components of $\Omega_1$ is no greater than $m_t$, and the total number of connected components of $\partial \Omega_1$ is exactly $C_t - 1 - l_t + m_t$.

            By induction, the claim holds for each component of $\Omega_1$, so the total number of connected stable minimal hypersurfaces with the area upper bound in $\Omega_1$ is at least
            \begin{equation}
                (C_t - 1 - l_t + m_t) - m_t = C_t - 1 - l_t\,,
            \end{equation}
            and thus, the number of those inside the metric completion of $\tilde N_t$ with the area upper bound is at least
            \begin{equation}
                (C_t - 1 - l_t) + l_t = C_t - 1\,.
            \end{equation}
        \end{proof}
    \end{proof}

    By Theorem \ref{thm:conf_minmax}, for any $P \in \mathbb{N}$ large enough, there exists a union of smooth, connected, closed, embedded, two-sided minimal hypersurfaces $\Gamma^P_1, \cdots, \Gamma^P_{l_P}$ contained in $N_{j_p}\backslash \partial N_{j_p}$ such that
        \begin{equation}
            \begin{aligned}
                \omega_P(\mathcal{C}(N)) &= \sum^{l_P}_{k = 1} \mathrm{Area}(\Gamma^P_k)\,,\\
                P &= \sum^{l_P}_{k = 1} \mathrm{index}(\Gamma^P_k)\,.
            \end{aligned}
        \end{equation}

    By Lemma \ref{lem:area}, we know that $\mathrm{Area}(\Gamma^P_k) \geq \mathrm{Area}(\Sigma^{j_p}_1) \geq p$ for any $k\,$. Moreover, by Lemma \ref{lem:num}, we also know that the number of minimal hypersurfaces in $\set{\Gamma^P_k}$ above with area no greater than $(p+1)\cdot \mathrm{Area}(\Sigma^{j_p}_1)$ is no greater than $2 \mathcal{N}(p) + 1$, and hence, we can obtain an upper bound on $l_P$, i.e.,
    \begin{equation}
        l_P \leq 2 \mathcal{N}(p) + 1 + \frac{\omega_P}{(p+1)\cdot \mathrm{Area}(\Sigma^{j_p}_1)} \leq 2 \mathcal{N}(p) + 1 + \frac{P}{p + 1} + \frac{\tilde C P^{\frac{1}{n+1}}}{p + 1}\,.
    \end{equation}
    
    By the pigeonhole principle, we know that there exists a $\Gamma^P_{k'}$ with
    \begin{equation}
        \begin{aligned}
            \mathrm{index}(\Gamma^P_{k'}) &\geq \frac{P}{l_P}\\
                & \geq \frac{P}{2 \mathcal{N}(p) + 1 + \frac{P}{p + 1} + \frac{\tilde C P^{\frac{1}{n+1}}}{p+1}}\\
                & = \frac{p+1}{\frac{2(p+1) \mathcal{N}(p) + p+1}{P} + 1 + \tilde C P^{-\frac{n}{n+1}}}\\
                & \geq p\,,
        \end{aligned}
    \end{equation}
    when $P$ is large enough. 
    
    In this case, we now can define $\Gamma_p := \Gamma^P_{k'}$ chosen above.
    
    $\,$\\

    In summary, in each case, we obtain a sequence of connected two-sided minimal hypersurfaces with
    \begin{equation}
        \lim_{p \rightarrow \infty} \mathrm{Area}(\Gamma_p) = +\infty\,, \quad \lim_{p \rightarrow \infty} \mathrm{index}(\Gamma_p) = +\infty\,.
    \end{equation}\\

\appendix

\section{X. Zhou's Multiplicity One Theorem}\label{app:A}
    
    The strategy of X. Zhou's proof lies upon the approximation of minimal hypersurfaces by a sequence of min-max prescribed mean curvature (PMC) hypersurfaces. Roughly speaking, in a bumpy closed Riemannian manifold, for each $p$-width, he constructs a sequence of PMC hypersurfaces converging to a $p$-width minimal hypersurface $\Sigma$. Analysis on such a convergence implies that if the multiplicity of the convergence is greater than one, then either one could get a non-trivial Jacobi field, or a positive function $\varphi$ satisfying the PDE,
    \begin{equation}
        L_{\Sigma} \varphi = \mathfrak{h}\,,
    \end{equation}
    where $L_\Sigma$ is the Jacobi operator and $\mathfrak{h}$ is the prescribed function. The first case can be excluded by the bumpiness while the second case is also excluded by choosing an appropriate $\mathfrak{h}$ such that the PDE has no positive solution. Therefore, one is able to show that the $p$-width minimal hypersurface should has multiplicity one. Note that the construction of min-max PMC hypersurfaces is based on the work by X. Zhou and J. Zhu \cite{zhou_existence_2018}.

    However, the results above are all in the setting of a closed Riemmanian manifold, and here we need to deal with $(N_\varepsilon, \partial N_\varepsilon)$ which is a Riemannian manifold with boundary. Therefore, one would expect to resolve this issue by directly adapting their methods in the free boundary min-max setting but it seems that the analysis on the touching set near the boundary, if possible, would require tedious work on this adaption. In fact, in our setting, the multiplicity one theorem could be proved without such a general theory.

    Firstly, for $\varepsilon>0$ small enough, we may view $(N_\varepsilon, \partial N_\varepsilon)$ as an subset of some closed smooth Riemannian manifold $\tilde N_\varepsilon$ and then choose $\mathfrak{h}$ defined in $\tilde N_\varepsilon$ such that all minimal hypersurfaces with area bounded by $(\omega_p(\mathcal{C}(N), h)+1)$ and Morse index bounded by $p$ contained in the interior of $N_\varepsilon$ (they are finitely many) satisfy the sign-changing property as in Lemma 4.2 in \cite{zhou_multiplicity_2019}. Meanwhile, $\mathfrak{h}\in \mathcal{S}(\tilde N_\varepsilon)\,$.

    Second, define a cut-off function $\eta$ supported in the interior of $N_\varepsilon$ such that
    \begin{equation}
        \eta(x) \equiv 1\,, \text{ if } \,\mathrm{dist}_{h_\varepsilon}(x, \partial N_\varepsilon) \geq 1\,,
    \end{equation}
    and define $\tilde{\mathfrak{h}} = \mathfrak{h} \cdot \eta$ on $N_\varepsilon$.

    Now, on $(N_\varepsilon, \partial N_\varepsilon)\,$, given any associated $(X, Z)$-homotopy class $\Pi$ of $\Phi_0$ with
    \begin{equation}
        \omega_p(\mathcal{C}(N),h) + 1 \geq \mathbf{L}(\Pi) > \max_{x \in Z} \mathbf{M}(\partial \Phi_0(x)),   
    \end{equation}
    we would follow the proof of Theorem 3.6 in \cite{zhou_multiplicity_2019} with prescribed mean curvature $\delta \cdot \tilde{\mathfrak{h}}$ where the number $\delta > 0$ is small enough. Note that for a pulled-tight min-max sequence $S = \{\Phi_l\} \in \Pi$, any varifold $V \in \mathbf{C}(S)$ is $\mathcal{A}^{\delta\tilde{\mathfrak{h}}}$-stationary. 

    Since $\mathfrak{h}\equiv 0$ near $\partial N_\varepsilon$, the regularity for almost-minimizing free boundary min-max minimal hypersurfaces in Li-Zhou \cite{li_min-max_2016} implies that there exists a $V_\delta \in \mathcal{C}(S)\,$, $\mathrm{spt}(V_\delta) \cap \set{x:\mathrm{dist}_{h_\varepsilon}(x, \partial N_\varepsilon) < 1/2}$ is either an empty set or a smooth hypersurface. Moreover, in the interior of $\set{x: \mathrm{dist}_{h_\varepsilon}(x, \partial N_\varepsilon) > 1}\,$, such a varifold $V_\delta$ is $\delta\mathfrak{h}$-almost minimizing in small annuli, since $\tilde{\mathfrak{h}}$ and $\mathfrak{h}$ coincide there. Hence, by the maximum principle by B. White \cite{white_maximum_2009} and monotonicity formulae, if we fix $\varepsilon$ small enough and take $\delta$ even smaller, $V_\delta$ should be contained in the interior of $\set{x: \mathrm{dist}_{h_\varepsilon}(x, \partial N_\varepsilon) > 1}\,$, since otherwise, the area of $V_\delta$ will be greater than $(\omega_p(\mathcal{C}(N),h) + 1)\,$. Thus, the proof of Theorem 3.6 in \cite{zhou_multiplicity_2019} implies that $V_\delta$ is a smooth, closed, multiplicity one, properly almost embedded hypersurface with Morse index bounded above by $p$ as boundary of a Cacciopoli set.

    Finally, we can follow Theorem 4.1 in \cite{zhou_multiplicity_2019}, taking $\delta \rightarrow 0$, and obtain an embedded closed multiplicity one minimal hypersurface with area $\omega_p(N_\varepsilon,h_\varepsilon)$ and Morse index bounded by $p$ in the interior of $(N_\varepsilon, \partial N_\varepsilon)$.

\bibliography{reference}
\end{document}